\documentclass[12pt,reqno]{amsart}
\usepackage{fullpage}
\usepackage{times}
\usepackage{amsmath,amssymb,amsthm,url}
\usepackage[utf8]{inputenc}
\usepackage[english]{babel}
\usepackage{comment}
\usepackage{bbm}
\usepackage{enumerate}
\usepackage{bm}
\usepackage{graphicx}
\usepackage{mathrsfs}
\usepackage[colorlinks=true, pdfstartview=FitH, linkcolor=blue, citecolor=blue, urlcolor=blue]{hyperref}
\usepackage{tabstackengine}
\stackMath
\numberwithin{equation}{section}

\newtheorem{thm}{Theorem}[section]
\newtheorem{lem}[thm]{Lemma}
\newtheorem{prop}[thm]{Proposition}
\newtheorem{cor}[thm]{Corollary}
\newtheorem{conj}[thm]{Conjecture}
\newtheorem{claim}[thm]{Claim}

\theoremstyle{definition}

\newtheorem{rem}[thm]{Remark}

\newenvironment{poc}{\begin{proof}[Proof of claim]}{\end{proof}}

\newcommand{\F}{\mathbb{F}}

\newcommand{\sgn}{\operatorname{sgn}}

\title{An absolute bound for generalized Diophantine tuples\\over polynomial rings}
\author{Kin Ming Tsang}
\address{Department of Mathematics \\ University of British Columbia\\ Vancouver  V6T 1Z2 \\ Canada}
\email{kmtsang@math.ubc.ca}
\author{Chi Hoi Yip}
\address{Department of Mathematics, Hong Kong University of Science and Technology, Clear Water Bay, Hong Kong}
\email{machyip@ust.hk}

\begin{document}

\begin{abstract}
Let $\mathbb F$ be an algebraically closed field of characteristic $0$. Let $k\geq 2$ be an integer, and let $n\in \mathbb F[x]\setminus\{0\}$. We study generalized Diophantine tuples $A\subset \mathbb F[x]$ with property $D_k(n)$, meaning that $ab+n$ is a $k$-th power in $\mathbb F[x]$ for all distinct elements $a,b\in A$. For $k\ge18$, we prove that every such tuple satisfies $|A|\le6$, except for the necessary exceptional family in which $n=s^2$ is a $k$-th power and $A\subset s\F$. This bound is absolute: it is independent of both $n$ and $\deg n$. Our proof develops a new method for studying polynomial Diophantine tuples,
combining a determinant criterion, generalizations of the Mason--Stothers theorem, and the Combinatorial Nullstellensatz. We also record a conditional analogue for generalized Diophantine tuples over the integers.
\end{abstract}
\maketitle

\section{Introduction}
A set $A$ of positive integers is a \emph{Diophantine tuple} if $ab+1$ is a perfect square for all distinct $a,b\in A$. For example, $\{1,3,8,120\}$ is a Diophantine $4$-tuple, discovered by Fermat. This example is now known to be optimal: after Dujella proved that there are only finitely many Diophantine quintuples \cite{D04} in 2004, He, Togb\'e, and Ziegler \cite{HTZ19} eventually proved that there is no Diophantine quintuple in 2019. The notion of Diophantine tuples has since been generalized in many directions and over many rings. For background and a comprehensive discussion, see Dujella's monograph \cite{D24}. 

A natural generalization is obtained by replacing the square $ab+1$ by
a shifted $k$-th power. Let $n,k$ be integers with $n\neq 0$ and $k\geq 2$.
A set $A$ of positive integers is called a \emph{Diophantine
tuple with property $D_k(n)$} if $ab+n$ is a $k$-th power of a nonnegative integer for all distinct $a,b\in A$. Following standard notation, define
\[
M_k(n)=\sup\{|A|: A\subseteq \mathbb N \text{ satisfies property }D_k(n)\}.
\]
These generalized Diophantine tuples have attracted considerable
attention; see, for example,
\cite{BCM22,BD03,BHP25,DKM22,D02,DF05,DL05,KYY,Y24,Y25,Y26}. The best known
unconditional upper bounds are of the form
\[
M_k(n)\ll_k \log |n|,
\]
with the implied constant depending on $k$; see \cite{KYY,Y24,Y25} for
the strongest known constants.

Under the Uniformity Conjecture (known to be a consequence of the Bombieri–Lang conjecture by Caporaso, Harris, and Mazur \cite{CHM}), it is well-known that for each fixed $k \geq 2$, there is a constant $C_k$ such that $M_k(n)\leq C_k$ holds for all nonzero integers $n$; see, for example, \cite{D02} and \cite[Remark 2.12]{CY26}. Recently Croot and Yip \cite[Theorem 2.10]{CY26} showed that, assuming the Lander--Parkin--Selfridge conjecture~\cite{LPS67} related to sums of powers (see Conjecture~\ref{conj:LPS}), $M_k(n)\leq 21738$ holds for all $k\geq 25$ and $n \neq 0$. Consequently, assuming both conjectures, and after increasing $C$ to handle the finitely many exponents $2\leq k\leq 24$, one obtains an absolute constant $C$ such that $M_k(n)\leq C$ for all $k\geq 2$ and all nonzero integers $n$. Thus, over the integers, absolute boundedness for generalized Diophantine tuples is currently known only conditionally.

The aim of this paper is to prove an unconditional polynomial-ring
analogue of the above conditional absolute boundedness phenomenon.
Polynomial Diophantine tuples were first studied by Jones
\cite{J76,J78}. Throughout the paper, $\F$ denotes an algebraically
closed field of characteristic $0$, $k$ denotes a positive integer at least $2$, and
$n\in \F[x]\setminus\{0\}$. A set $A\subset \F[x]$ is called a \emph{Diophantine tuple with property $D_k(n)$} if the polynomial $ab+n$ is a $k$-th power in $\F[x]$ for all distinct $a,b\in A$. Here we view the zero polynomial as a $k$-th power. These generalized polynomial Diophantine tuples are well-studied. 

Several cases of this polynomial problem were previously understood.
When $n$ is a nonzero constant, one may reduce to the case $n=1$ by
scaling, since $\F$ is algebraically closed. For Diophantine tuples
$A\subset \F[x]$ with property $D_k(1)$ and $A\not\subset \F$, strong
bounds are known. Dujella and Jurasi\'c \cite{DJ10} proved that
$|A|\leq 7$ when $k=2$, while Dujella and Luca \cite{DL07} proved that
$|A|\leq 5$ for $k=3$, $|A|\leq 4$ for $k=4$, $|A|\leq 3$ for
$k\geq 5$, and $|A|\leq 2$ for even $k\geq 5$. The cases where $n$ is
linear or quadratic have also been studied; see, for example,
\cite{DFT02,DFW06} for the linear case and \cite{DFW06,DJ11,J11} for the
quadratic case.

By contrast, there seems to be no systematic study of upper bounds when the degree of $n$ is at least $3$ \footnote{Private communication with Andrej Dujella.}. This is consistent with the usual integer--function field analogy: the best known unconditional bounds over the integers are logarithmic in $|n|$, and $\deg n$ is the natural function-field analogue of $\log |n|$. Thus one expects the problem to become more difficult as $\deg n$ grows. For instance, when $d\geq 3$, it remains open to give an upper bound depending only on $d$ and $k$ for the size of polynomial Diophantine tuples $A\subset \F[x]$ with property $D_k(n)$ and $\deg n=d$; see \cite[Problem 1.14]{Dopen}.

Our main results give an absolute bound, independent of $n$ and its
degree, for all sufficiently large $k$. Thus, for sufficiently large exponents, our theorem shows that the expected dependence
on $\deg n$ disappears entirely.

\begin{thm}\label{thm:main}
Let \(\F\) be an algebraically closed field of characteristic \(0\), let
\(k\geq 18\), and let \(n\in \F[x]\setminus\{0\}\). Let \(A\subset \F[x]\)
be a Diophantine tuple with property \(D_k(n)\). If \(n=s^2\) for some \(s\in \mathbb F[x]\) and \(n\) is a \(k\)-th power in
\(\mathbb F[x]\), assume additionally that
\[
        A\not\subset s\mathbb F:=\{cs:c\in\mathbb F\}.
\]
Then
\[
        |A|\leq 6.
\]
Moreover, if \(n\) is nonsquare and \(k\) is even, then
\[
        |A|\leq 5.
\]
\end{thm}

When \(n=s^2\) is also a \(k\)-th power, the additional assumption
\(A\not\subset s\F\) is necessary. Indeed, if $n=s^2$ is a $k$-th power and $A\subset s\F$,
then no finite bound is possible since $s\F$ itself is an
infinite Diophantine tuple with property $D_k(n)$: 
for any $c_1,c_2\in \F$,
\[
c_1s \cdot c_2s\ +n=n(c_1c_2+1),
\]
which is again a $k$-th power in $\F[x]$. 

A key feature of the paper is a new selection framework for polynomial
Diophantine tuples. An important ingredient is the following proposition.

\begin{prop}\label{prop:18}
Let \(n\in\F[x]\setminus\{0\}\), and let \(A\subset\F[x]\) be a Diophantine
tuple with property \(D_k(n)\), where \(k\ge18\). There do not exist pairwise
distinct $a_1,a_2,a_3,b_1,b_2,b_3\in A$
such that
\[
 a_i b_j+n\neq0\qquad(1\le i,j\le3)
\]
and
\[
 \frac{(a_1b_1+n)(a_2b_2+n)}
      {(a_1b_2+n)(a_2b_1+n)}\notin\F^\times.
 \]
\end{prop}

Proposition~\ref{prop:18} is the structural core of the proof.  It rules out a
``generic'' \(3\times 3\) configuration inside a generalized Diophantine tuple:
namely, six elements \(a_1,a_2,a_3,b_1,b_2,b_3\) for which all products
\(a_i b_j+n\) are nonzero and one natural cross-ratio is nonconstant.  Its proof
combines determinant identities arising from the Diophantine condition with
function-field \(abc\)-type estimates, in particular generalizations of the
Mason--Stothers theorem.

The proof of Theorem~\ref{thm:main} then separates into two parts.  First, we
establish Proposition~\ref{prop:18}.  Second, we prove selection results showing
that every tuple of size at least \(7\), unless it belongs to the evident
exceptional family, must contain such a generic \(3\times 3\) configuration.
The nonsquare case is handled by a valuation argument.  In the square case, this
valuation argument no longer suffices; it would only give a bound linear in
\(\deg n\).  We therefore use an additional geometric selection argument, together
with the Combinatorial Nullstellensatz to handle the case where the tuple lies in
an affine \(\F\)-line.

The same underlying ideas also give a conditional improvement for generalized Diophantine tuples over the integers.  In Section~\ref{sec:integer-warmup}, we prove that, assuming the Lander--Parkin--Selfridge conjecture, one has
\[
M_k(n)\le 5
\]
for all \(k\ge 7\) and all nonzero integers \(n\).  This significantly improves the aforementioned result of Croot and Yip~\cite{CY26}, which gives \(M_k(n)\le 21738\) for all \(k\ge 25\) under the same conjecture.

\medskip

\textbf{Organization of the paper.} In Section~\ref{sec:integer-warmup}, as a warm-up, we prove a conditional analogue for generalized Diophantine tuples over the integers. In Section~\ref{sec:prelim}, we list some preliminary tools and prove some auxiliary results. In Section~\ref{sec:framework}, we prove Proposition~\ref{prop:18}. In Section~\ref{sec:thm1}, we prove Theorem~\ref{thm:main} when $n$ is a nonsquare. In Section~\ref{sec:thm2}, we prove Theorem~\ref{thm:main} when $n$ is a square. 

\section{A warm-up: a conditional integer analogue}
\label{sec:integer-warmup}

We start with a conditional analogue over the integers, which serves as a simple
model case for the polynomial argument.  Its proof is considerably simpler, but
it follows the same broad outline: a low-rank relation among shifted products
produces a short equality of \(k\)-th powers, which is then ruled out by the
Lander--Parkin--Selfridge conjecture.  In the polynomial setting, the analogous
obstruction requires the function-field and selection arguments developed later.

We first recall the following conjecture of Lander, Parkin, and Selfridge \cite{LPS67}.

\begin{conj}[Lander--Parkin--Selfridge conjecture]\label{conj:LPS}
Let \(r,s,k\) be positive integers. If
\[
        \sum_{i=1}^{r} a_i^k=\sum_{j=1}^{s} b_j^k,
\]
where \(a_1,\ldots,a_r,b_1,\ldots,b_s\) are positive integers such that
\(a_i\neq b_j\) for all \(1\le i\le r\) and \(1\le j\le s\), then
\[
        r+s\ge k.
\]
\end{conj}

\begin{thm}\label{thm:integer}
Let \(k\ge 7\) and let \(n\in \mathbb Z\setminus\{0\}\). Assume
Conjecture~\ref{conj:LPS}. Then
\[
        M_k(n)\le 5.
\]
\end{thm}

\begin{proof}
Suppose, for contradiction, that there is a set \(A\subset \mathbb N\) with
property \(D_k(n)\) and \(|A|\ge 6\). Choose six distinct elements of \(A\),
denoted
\[
        a_1,a_2,a_3,b_1,b_2,b_3.
\]
Since these six elements are distinct, for every \(1\le i,j\le 3\) we may write
\[
        a_i b_j+n=x_{ij}^k,
        \qquad x_{ij}\in \mathbb Z_{\ge 0}.
\]
The matrix \(M=(a_i b_j+n)_{1\le i,j\le 3}\) has rank at most \(2\), since its
\(j\)-th column is
\[
        b_j(a_1,a_2,a_3)^T+n(1,1,1)^T.
\]
Thus 
\[
        \det(M)=\det(x_{ij}^k)_{1\le i,j\le 3}=0.
\]
For \(\sigma\in S_3\), put
\[
        X_\sigma=\prod_{i=1}^3 x_{i,\sigma(i)}.
\]
Expanding the determinant gives
\begin{equation}\label{eq:integer-det}
        \sum_{\sigma\in A_3} X_\sigma^k
        =
        \sum_{\sigma\notin A_3} X_\sigma^k,
\end{equation}
where \(A_3\) denotes the alternating subgroup of \(S_3\).

We first show that not all \(X_\sigma\) vanish. Observe that no two zero
entries of the matrix \(M\) can lie in the same row
or in the same column. Indeed,
\[
        a_i b_j+n=a_i b_{j'}+n=0
\]
would imply \(a_i(b_j-b_{j'})=0\), impossible; the column case is similar.
Hence there exists a permutation \(\sigma\in S_3\) avoiding all zero entries,
and for this permutation we have \(X_\sigma>0\). Therefore the common value
of the two sides of equation~\eqref{eq:integer-det} is positive. Thus, after deleting
the zero terms, both sides of equation~\eqref{eq:integer-det} contain at least one
positive term.

We next show that no positive base occurs on both sides of
equation~\eqref{eq:integer-det}. Define
\[
        P_\sigma=X_\sigma^k
        =
        \prod_{i=1}^3 (a_i b_{\sigma(i)}+n).
\]
Let \(\sigma,\tau\in S_3\) have opposite signs, and suppose
\(X_\sigma X_\tau\ne 0\). Then \(\sigma^{-1}\tau\) is a transposition. Thus
there exist distinct \(r,s\in\{1,2,3\}\), with remaining index \(t\), such
that
\[
        \tau(r)=\sigma(s),\qquad
        \tau(s)=\sigma(r),\qquad
        \tau(t)=\sigma(t).
\]
Writing \(u=\sigma(r)\) and \(v=\sigma(s)\), we obtain
\[
\begin{aligned}
P_\sigma-P_\tau
&=
(a_t b_{\sigma(t)}+n)
\Bigl((a_r b_u+n)(a_s b_v+n)
      -(a_r b_v+n)(a_s b_u+n)\Bigr)  \\
&=
n(a_r-a_s)(b_u-b_v)(a_t b_{\sigma(t)}+n).
\end{aligned}
\]
This is nonzero, since \(n\ne0\), the chosen elements are pairwise distinct, and \(X_\sigma\ne0\) implies \(a_t b_{\sigma(t)}+n\ne0\). Therefore \(P_\sigma\ne P_\tau\), and hence \(X_\sigma\ne X_\tau\), whenever \(\sigma\) and \(\tau\) have opposite signs and the corresponding terms are positive.

Consequently, after deleting the zero terms from equation~\eqref{eq:integer-det}, we obtain a nontrivial equality of positive integer \(k\)-th powers with no common base on the two sides. The total number of terms is at most \(6\). By Conjecture~\ref{conj:LPS}, this total number must be at least \(k\), which is impossible since \(k\ge 7\). This contradiction proves the theorem.
\end{proof}

It is worth noting that the later polynomial argument is not obtained by applying a direct function-field analogue of Conjecture~\ref{conj:LPS}. Such an analogue is already false over an algebraically closed field. Indeed, since \(\mathbb F\) is
algebraically closed, if \(\zeta\ne 1\) is a \(k\)-th root of unity and \(f\in\mathbb F[x]\setminus\{0\}\), then $f^k=(\zeta f)^k$ although \(f\ne \zeta f\).  Thus equalities of \(k\)-th powers over
\(\mathbb F[x]\) have unavoidable degeneracies.  In the polynomial setting, the role of Conjecture~\ref{conj:LPS} will instead be played by function-field estimates, together with the selection arguments developed below.

A similar caveat applies to the remaining exponents \(2\le k\le 17\).  Although uniformity results for rational points can be used in the integer setting for fixed \(k\) (see for example \cite[Remark 2.12]{CY26}), function-field uniformity results generally have to account for isotrivial phenomena and depend on a degeneracy locus; see Caporaso~\cite{C02}. Thus, it appears that they do not directly give the degree-independent bounds needed here.

\section{Preliminaries}\label{sec:prelim}

\subsection{\(v_{\alpha}\)-valuations}
For \(\alpha\in \F\) and \(f\in \F[x]\), we denote by \(v_\alpha(f)\) the order of vanishing of \(f\) at \(x=\alpha\). Equivalently, \(v_\alpha(f)=m\) if \(m\) is the largest nonnegative integer for which
\[
    f=(x-\alpha)^m g
\]
for some \(g\in \F[x]\). We use the convention \(v_\alpha(0)=\infty\).

Next we record an important property that follows from the definition. Let \(\alpha\in \F\) and \(f,g\in \F[x]\). Then
\[
    v_{\alpha}(f+g)\geq \min\{v_{\alpha}(f),v_{\alpha}(g)\}.
\]
Moreover, equality holds whenever \(v_{\alpha}(f)\neq v_{\alpha}(g)\).

\begin{lem}\label{lem:valuation-ratio}
Let \(\alpha\in\F\). Let \(F_0,E_1,E_2\in\F[x]\), with \(F_0\neq 0\), and suppose
that
\[
    v_\alpha(E_1)>v_\alpha(F_0),
    \qquad
    v_\alpha(E_2)>v_\alpha(F_0).
\]
If $E_1\neq E_2$, then
\[
    \frac{F_0+E_1}{F_0+E_2}\notin \F^\times.
\]
\end{lem}

\begin{proof}
Suppose otherwise that
\[
    F_0+E_1=C(F_0+E_2)
\]
for some \(C\in\F^\times\). Since $E_1\neq E_2$, we have $C\neq 1$ and 
\[
    (1-C)F_0=CE_2-E_1.
\]
However, the left-hand side has \(v_\alpha\)-valuation \(v_\alpha(F_0)\), while the
right-hand side has valuation strictly larger than \(v_\alpha(F_0)\), a
contradiction. 
\end{proof}

\subsection{Combinatorial Nullstellensatz}\label{subsec:CN}
The Combinatorial Nullstellensatz, introduced by Alon in 1999 \cite{A99}, is a powerful algebraic technique for proving existence results in combinatorics, number theory, and finite geometry.
We record the following version of Combinatorial Nullstellensatz \cite[Lemma 2.1]{A99}.
\begin{lem}
[Combinatorial Nullstellensatz] \label{lem:CN}
Let \(P=P(x_1,x_2,\dots,x_m)\) be a polynomial in \(m\) variables over an arbitrary field \(K\). Suppose that for each \(1\le i\le m\), the degree of \(P\) as a polynomial in \(x_i\) is at most \(t_i\), and let \(S_i\subseteq K\) be a set of at least \(t_i+1\) distinct elements of \(K\). If
\[
P(x_1,x_2,\dots,x_m)=0
\]
for all \((x_1,\dots,x_m)\in S_1\times S_2\times\cdots\times S_m\), then $P$ is the zero polynomial.
\end{lem}

First, we use the Combinatorial Nullstellensatz over $\F(x)$ to prove the following lemma. 

\begin{lem}
\label{lem:affine-7}
Let $\F$ be a field with characteristic zero. Let \(\Lambda\subset \F\) satisfy
\(|\Lambda|\geq 7\). Let \(p,q\in \F(x)\), with \(q\neq 0\), and suppose
that \(p\) and \(q\) are not both constant. For variables \(\lambda,\mu\), consider the polynomial
\[
        B(\lambda,\mu)=(p+\lambda q)(p+\mu q)+1\in \F(x)[\lambda,\mu].
\]
Then there exist pairwise distinct
\[
        \lambda_1,\lambda_2,\mu_1,\mu_2\in \Lambda
\]
such that
\[
        B(\lambda_i,\mu_j)\neq 0\qquad (1\leq i,j\leq 2),
\]
and such that
\[
        \frac{B(\lambda_1,\mu_1)B(\lambda_2,\mu_2)}
             {B(\lambda_1,\mu_2)B(\lambda_2,\mu_1)}
\]
is nonconstant.
\end{lem}

\begin{proof}
Let \(X_1,X_2,Y_1,Y_2\) be independent variables. For $1\leq i,j\leq 2$, set
\[
        \mathcal B_{ij}=B(X_i,Y_j) \in \F(x)[X_1,X_2,Y_1,Y_2].
\]

Put
\[
        \mathcal N=\mathcal B_{11}\mathcal B_{22},
        \qquad
        \mathcal M=\mathcal B_{12}\mathcal B_{21},
\]
and define
\[
        \mathcal W=\mathcal N'\mathcal M-\mathcal N\mathcal M',
\]
where the prime denotes differentiation with respect to \(x\). For each of the variables \(X_1,X_2,Y_1,Y_2\), the polynomials
\(\mathcal N,\mathcal M,\mathcal N'\), and \(\mathcal M'\) have degree at most
\(1\) in that variable. Hence $\mathcal W$ has degree at most \(2\) in each variable.

\begin{claim}
\(\mathcal W\neq 0\) as an element of $\F(x)[X_1,X_2,Y_1,Y_2].$
\end{claim}
\begin{poc}
For $1\leq i,j\leq 2$,
\[
        U_i=p+X_iq,\qquad V_j=p+Y_jq.
\]
Then
\[
\begin{aligned}
        \mathcal N-\mathcal M
        &=(U_1V_1+1)(U_2V_2+1)
          -(U_1V_2+1)(U_2V_1+1)  \\
        &=(U_1-U_2)(V_1-V_2) 
        =q^2(X_1-X_2)(Y_1-Y_2).
\end{aligned}
\]
Suppose, for contradiction, that \(\mathcal W=0\). Then
\[
        \left(\frac{\mathcal N}{\mathcal M}\right)'=\frac{\mathcal W}{\mathcal M^2}=0.
\]
Since
\[
        \frac{\mathcal N}{\mathcal M}-1
        =
        \frac{\mathcal N-\mathcal M}{\mathcal M}
        =
        \frac{q^2(X_1-X_2)(Y_1-Y_2)}{\mathcal M},
\]
and \((X_1-X_2)(Y_1-Y_2)\) is independent of \(x\), it follows that
\begin{equation}\label{eq:derivative}
        \left(\frac{\mathcal M}{q^2}\right)'=0.
\end{equation}
Now compare two coefficients in the identity~\eqref{eq:derivative}. Note that we have
\[
        B(X,Y)=q^2XY+pq(X+Y)+p^2+1,
\]
and
\[
        \frac{\mathcal M}{q^2}
        =
        \frac{B(X_1,Y_2)B(X_2,Y_1)}{q^2}.
\]
Thus, the coefficient of the monomial \(X_1X_2Y_1Y_2\) in
\(\mathcal M/q^2\) is \(q^2\) and the coefficient of the monomial \(X_1X_2Y_1\) in
\(\mathcal M/q^2\) is \(pq\). It follows that
\[
        (q^2)'=0, \qquad (pq)'=0.
\]
Since \(q\neq 0\), the first identity gives
\(q'=0\). The second identity then gives \(p'q=0\), hence \(p'=0\).
Thus \(p\) and \(q\) are both constant, contradicting the hypothesis.
\end{poc}

Now define
\[
        \Psi
        =
        \mathcal W \cdot 
        \prod_{i=1}^2\prod_{j=1}^2 (X_i-Y_j)
        \cdot \prod_{i=1}^2\prod_{j=1}^2 \mathcal B_{ij}
        \in \F(x)[X_1,X_2,Y_1,Y_2].
\]
This is a nonzero polynomial. Moreover, each variable has degree at most \(6\)
in \(\Psi\): the factor \(\mathcal W\) contributes degree at most \(2\), the
factor \(\prod(X_i-Y_j)\) contributes degree \(2\), and the factor
\(\prod \mathcal B_{ij}\) contributes degree at most \(2\). Thus
\[
        \deg_{X_1}\Psi,\deg_{X_2}\Psi,
        \deg_{Y_1}\Psi,\deg_{Y_2}\Psi\leq 6.
\]
Since \(|\Lambda|\geq 7\), Lemma~\ref{lem:CN} gives a point
\[
        (\lambda_1,\lambda_2,\mu_1,\mu_2)\in \Lambda^4
\]
such that
\[
        \Psi(\lambda_1,\lambda_2,\mu_1,\mu_2)\neq 0.
\]

After this specialization, write
\[
        B_{ij}=B(\lambda_i,\mu_j),\qquad
        N=B_{11}B_{22},\qquad
        M=B_{12}B_{21}.
\]
The nonvanishing of \(\Psi\) gives
\[
        B_{ij}\neq 0,\qquad \lambda_i\neq\mu_j,\qquad N'M-NM'\neq 0
        \qquad (1\leq i,j\leq 2).
\]
If \(\lambda_1=\lambda_2\), then \(N=M\), and hence \(N'=M'\), a contradiction.
Thus \(\lambda_1\neq\lambda_2\). The same argument gives
\(\mu_1\neq\mu_2\). Hence the four chosen elements are pairwise distinct.
Finally,
\[
        \left(\frac{N}{M}\right)'
        =
        \frac{N'M-NM'}{M^2}
        \neq 0,
\]
since \(M\neq 0\). Therefore \(N/M\) is nonconstant, as required.
\end{proof}

Next we deduce the following corollary.
\begin{cor}
\label{cor:affine}
Let \(s\in \F[x]\setminus\{0\}\), and let \(A\subset \F[x]\). Suppose that $A \not \subset s\F$,
\[
 |A|\geq 7,
\]
 and
\[
A \subset \{a_0+\lambda h: \lambda \in \F\}
\]
for some $a_0\in \F[x]$ and \(h\in \F[x]\setminus \{0\}\). Then there exist pairwise distinct
\[
 a_1,a_2,b_1,b_2\in A
\]
such that
\[
 a_i b_j+s^2\neq 0\qquad (1\leq i,j\leq 2),
\]
and
\[
 \frac{(a_1b_1+s^2)(a_2b_2+s^2)}
        {(a_1b_2+s^2)(a_2b_1+s^2)}
 \notin \F^\times.
\]
\end{cor}
\begin{proof}
We can write
\[
 A=\{a_0+\lambda h:\lambda\in\Lambda\},
\]
where \(\Lambda\) is some subset of $\F$ with
\[
 |\Lambda|=|A|\geq 7.
\]
Define $p,q\in \F(x)$ by
\[
 p=a_0/s,\qquad q=h/s.
\]
Then \(q\neq 0\). Since \(A\not \subset s\F\), \(p\) and \(q\) are not both constant. 

Consider the polynomial
\[
 B(\lambda,\mu)=(p+\lambda q)(p+\mu q)+1 \in \F(x) [\lambda, \mu].
\]
By Lemma~\ref{lem:affine-7}, there exist pairwise distinct
\[
 \lambda_1,\lambda_2,\mu_1,\mu_2\in \Lambda
\]
such that
\[
 B(\lambda_i,\mu_j)\neq 0\qquad (1\leq i,j\leq 2),
\]
and
\[
 \frac{B(\lambda_1,\mu_1)B(\lambda_2,\mu_2)}
        {B(\lambda_1,\mu_2)B(\lambda_2,\mu_1)}
\]
is nonconstant. For $1\leq i,j\leq 2$, set
\[
 a_i=a_0+\lambda_i h, \qquad 
 b_j=a_0+\mu_j h.
\]
Since
\[
a_ib_j+s^2=(a_0+\lambda_i h)(a_0+\mu_j h)+s^2=s^2B(\lambda_i,\mu_j)
\]
for $1\leq i,j\leq 2$, the four elements $a_1,a_2,b_1,b_2$ are as required.
\end{proof}

This corollary will handle the square case when $A$ lies in an affine $\F$-line.

\subsection{An extension lemma}

In view of Proposition~\ref{prop:18}, the following extension lemma is useful. Indeed, to apply Proposition~\ref{prop:18}, we proceed with proof by contradiction. Suppose $A$ is a generalized Diophantine tuple with property $D_k(n)$ and it is large. First, we find pairwise distinct $a_1,a_2,b_1,b_2 \in A$ such that 
\[
 \frac{(a_1b_1+n)(a_2b_2+n)}
      {(a_1b_2+n)(a_2b_1+n)}\notin\F^\times,
 \]
and then extend it by adjoining two more elements of $A$ using the following extension lemma. Doing so would contradict Proposition~\ref{prop:18}.

\begin{lem}
\label{lem:extension-no-C}
Let \(n\in \F[x]\setminus\{0\}\), let \(A\subset \F[x]\) satisfy
\[
 |A|\geq 7,
\]
and let \(a_1,a_2,b_1,b_2\in A\) be pairwise distinct. Suppose that
\[
 a_i b_j+n\neq 0\qquad (1\leq i,j\leq 2).
\]
Then there exist \(a_3,b_3\in A\) such that
\[
 a_1,a_2,a_3,b_1,b_2,b_3
\]
are pairwise distinct, and
\[
 a_i b_j+n\neq 0\qquad (1\leq i,j\leq 3).
\]
\end{lem}

\begin{proof}
Put
\[
 R=A\setminus\{a_1,a_2,b_1,b_2\}.
\]
Then \(|R|\ge3\). Let
\[
 B=\{r\in R: rb_1+n=0\text{ or }rb_2+n=0\},
 \qquad
 C=\{r\in R: a_1r+n=0\text{ or }a_2r+n=0\}.
\]
Clearly, we have \(|B|,|C|\le2\), and \(B\cap C=\varnothing\).

It remains to choose \(z\in R\setminus B\) and \(w\in R\setminus C\) with
\(z\neq w\) and \(zw+n\neq0\). If \(B\neq\varnothing\), choose
\(w\in B\) and \(z\in R\setminus B\) as follows. When \(|B|=1\), fix
\(w\in B\); since \(|R\setminus B|\ge2\), one can choose \(z\in R\setminus B\)
with \(zw+n\neq0\). When \(|B|=2\), choose any \(z\in R\setminus B\), and then
choose \(w\in B\) with \(zw+n\neq0\). In both cases \(w\notin C\), because
\(B\cap C=\varnothing\), and \(z\neq w\).

If \(B=\varnothing\) but \(C\neq\varnothing\), the symmetric argument gives
\(z\in C\subset R\setminus B\) and \(w\in R\setminus C\) with
\(zw+n\neq0\). Finally, if \(B=C=\varnothing\), fix \(w\in R\); among the at
least two elements of \(R\setminus\{w\}\), at most one can satisfy
\(zw+n=0\). Choose any other one as \(z\).

Set \(a_3=z\) and \(b_3=w\). Then \(z,w\in R\) and \(z\neq w\), so the six
chosen elements are pairwise distinct. The conditions \(z\notin B\),
\(w\notin C\), and \(zw+n\neq0\), together with the original hypotheses, give
\(a_i b_j+n\neq0\) for all \(1\le i,j\le3\).
\end{proof}

\section{Proof of Proposition~\ref{prop:18}}\label{sec:framework}
In this section, we prove Proposition~\ref{prop:18}. To prove the proposition, we first convert the Diophantine condition
into determinant relations among shifted products, then use a Mason--Stothers type theorem to force vanishing of certain sums. 

\subsection{Applications of Mason--Stothers type theorems} In this subsection, we use Mason--Stothers type theorems to study equations arising from Diophantine tuples.

The Mason--Stothers theorem \cite{M84,S81}, also known as the ABC theorem over function fields, is a fundamental tool in the study of polynomial equations and Diophantine problems over function fields. It has inspired many extensions and refinements; see, for example, Brownawell--Masser \cite{BM86}, Shapiro--Sparer \cite{SS94}, Hu--Yang \cite{HY02}, Vaserstein--Wheland \cite{VW03}, and Croot--Hart \cite{CH10}.

Here we record a generalization by Vaserstein--Wheland \cite[Theorem 2.2(a)]{VW03}.

\begin{lem}[Vaserstein--Wheland]\label{lem:VW}
Let \(\F\) be an algebraically closed field of characteristic \(0\), let \(m\geq 2\), and let
\(y_0,y_1,\ldots,y_m\in \F[x]\setminus\{0\}\). Suppose that
\[
        y_1+\cdots+y_m=y_0,
        \qquad
        \gcd(y_1,\ldots,y_m)=1,
\]
that not all of \(y_0,y_1,\ldots,y_m\) are constant, and that no nonempty
subsum of \(y_1,\ldots,y_m\) vanishes. Then
\[
        \deg y_0
        <
        (m-1)\sum_{j=0}^{m}\nu(y_j),
\]
where \(\nu(y)\) denotes the number of distinct zeros of \(y\).
\end{lem}

Next, we use Lemma~\ref{lem:VW} to deduce a corollary that is more convenient to apply.

\begin{cor}\label{cor:VW}
Let \(\F\) be an algebraically closed field of characteristic \(0\).
Let \(m\geq 3\). Let \(p_1,\ldots,p_m\in \F[x]\setminus\{0\}\), and let
\(c_1,\ldots,c_m\in \F\), not all zero, such that \[\sum_{i=1}^m c_i p_i^k=0.\] Assume that for some \(i_0\in\{1,\ldots,m\}\), we have \(c_{i_0}\neq 0\) and $p_{i_0}/p_j \notin \F$ for all $j\neq i_0$. Then \[k<m(m-2).\]
\end{cor}

\begin{proof}
By relabeling the indices, we may assume \(i_0=1\). Choose a non-trivial relation involving
\(p_1^k\) with the smallest possible number of terms:
\[
        \lambda_1p_{i_1}^k+\lambda_2p_{i_2}^k\cdots+\lambda_s p_{i_s}^k=0,
\]
where $1=i_1<i_2<i_3\cdots<i_s\leq m$ and $\lambda_1, \ldots, \lambda_s\in \F^\times$.  
Then \(s\geq 3\): the cases \(s=1\) and \(s=2\) are impossible, since
\(s=2\) would imply \((p_1/p_{i_2})^k\in\F\), and hence \(p_1/p_{i_2}\in\F\), as
\(\F\) is algebraically closed. Moreover, by the minimality of \(s\), no nonempty proper subsum of
\[
        \lambda_1p_{i_1}^k+\lambda_2p_{i_2}^k\cdots+\lambda_s p_{i_s}^k
\]
vanishes; otherwise, either that subsum or its complement would give a
shorter relation still involving \(p_1^k\).

Let $d=\gcd(p_{i_1},\ldots,p_{i_s})$ and $q_j=p_{i_j}/d$ for $1\leq j \leq s$. 
Then
\[
        \lambda_1q_1^k+\cdots+\lambda_s q_s^k=0,
        \qquad
        \gcd(q_1,\ldots,q_s)=1.
\]
The \(q_i\)'s are not all constant; otherwise \(p_1/p_{i_j}=q_1/q_j\in\F\)
for every \(2\leq j\leq s\), contradicting the hypothesis.

Put
\[
        T=\max_{1\leq i\leq s}\deg q_i \geq 1.
\]
After relabelling the indices, we may assume \(\deg q_s=T\). Set
\[
        y_0=-\lambda_s q_s^k,
        \qquad
        y_i=\lambda_iq_i^k
        \quad (1\leq i\leq s-1).
\]
Then
\[
        y_1+\cdots+y_{s-1}=y_0.
\]
Moreover, no nonempty subsum of \(y_1,\ldots,y_{s-1}\) vanishes. Also,
\[
        \gcd(y_1,\ldots,y_{s-1})=1,
\]
because any common divisor of \(y_1,\ldots,y_{s-1}\) also divides
\(y_0\), and hence divides all \(q_i^k\).

Applying Lemma~\ref{lem:VW} with \(M=s-1\), we obtain
\[
        kT=\deg y_0
        <
        (s-2)\sum_{i=1}^{s}\nu(q_i^k)\leq (s-2)\sum_{i=1}^{s} \deg(q_i)\leq s(s-2)T.
\]
It follows that $k<s(s-2)\leq m(m-2)$, as required. 
\end{proof}

We also need a stronger version of Lemma~\ref{lem:VW} from Vaserstein--Wheland \cite{VW03}. To state it, we need to introduce some necessary terminology. Following \cite[Section 3]{VW03}, for any integer \(d\geq 1\) and any sequence
\(k_1,\ldots,k_d\) of integers, we define its \emph{diversity} $\Delta(k_1,\ldots,k_d)$ as follows:
\begin{equation}\label{eq:Delta}
        \Delta(k_1,\ldots,k_d)
        =
        \frac{d(d-1)}{2}
        -
        \min_{(\ell_1,\ldots,\ell_d)\in\mathcal X}
        \sum_{j=1}^d(\ell_j-k_j),
\end{equation}
where \(\mathcal X\) is the set of all
sequences \(\ell_1,\ldots,\ell_d\) of pairwise distinct integers such that $\ell_j\geq k_j$ for $1\leq j \leq d$.

Now let \(y_1,\ldots,y_d\in \F[x]\) be linearly independent over \(\F\), and put
\[
        y_0=y_1+\cdots+y_d.
\]
For each \(\alpha\in \F\), let
\[
        k_0\leq k_1\leq \cdots \leq k_d
\]
be the numbers
\[
        v_\alpha(y_0), v_{\alpha}(y_1), \ldots, v_{\alpha}(y_d)
\]
written in nondecreasing order. We define
\[
        \Delta_\alpha(y_1,\ldots,y_d)
        =
        \Delta(k_1,\ldots,k_d).
\]
Similarly, let
\[
        k_0\leq k_1\leq \cdots \leq k_d
\]
be the numbers
\[
        v_\infty(y_j):=-\deg(y_j)\qquad (0\leq j\leq d)
\]
written in nondecreasing order. We define
\[
        \Delta_\infty(y_1,\ldots,y_d)
        =
        \Delta(k_1,\ldots,k_d).
\]

It is immediate from the definition~\eqref{eq:Delta} that
\[
        0\leq \Delta(k_1,\ldots,k_d)\leq \binom d2.
\]
Consequently,
\[
        0\leq \Delta_\alpha(y_1,\ldots,y_d)\leq \binom d2
\]
for every \(\alpha\in \F \cup \{\infty\}\). Also, note that for all but finitely many \(\alpha \in \F \cup \{\infty\}\), we have \[\Delta_\alpha(y_1,\ldots,y_d)=0.\]

The following lemma is from \cite[Theorem 3.1]{VW03}.

\begin{lem}[Vaserstein--Wheland, linearly independent form]
\label{lem:VW-independent}
Let \(\F\) be an algebraically closed field of characteristic \(0\), let \(m\geq 2\), and let
$y_1,\ldots,y_m\in \F[x]\setminus\{0\}$
be linearly independent over \(\F\). Put
\[
 y_0=y_1+\cdots+y_m,
\]
and assume
\[
 \gcd(y_1,\ldots,y_m)=1.
\]
Then
\[
 \deg y_0
 \leq
 -m(m-1)
 +
 \sum_{\alpha\in \F \cup \{\infty\}}
 \Delta_\alpha(y_1,\ldots,y_m).
\]
\end{lem}

Next, we apply Lemma~\ref{lem:VW-independent} to a system of equations of a special form. Such a system will appear in our proof.

\begin{prop}
\label{prop:six-singleton}
Let \(q_1,\ldots,q_6\in \F[x]\setminus\{0\}\) be pairwise nonproportional over \(\F^\times\).
Suppose
\[
 q_1^k+q_2^k+q_3^k-q_4^k-q_5^k-q_6^k=0
\]
and there exists $\lambda \in \F^\times$ such that
\[
 q_1q_2q_3=\lambda q_4q_5q_6.
\]
Then
\[
 k<18.
\]
\end{prop}

\begin{proof}
Assume otherwise that $k\geq 18$. Dividing all \(q_i\) by their common gcd, we may assume without loss of
generality that
\[
        \gcd(q_1,\ldots,q_6)=1.
\]
Note that the product relation is preserved after this division.

Let
\[
        T=\max_{1\leq i\leq 6}\deg q_i.
\]
Since the \(q_i\)'s are pairwise nonproportional, not all of them are constant;
hence \(T\ge1\). Choose an index \(i_0\) such that \(\deg q_{i_0}=T\). Write
\[
        \varepsilon_i=
        \begin{cases}
        1,&1\leq i\leq 3,\\
        -1,&4\leq i\leq 6.
        \end{cases}
\]
Then
\[
        \sum_{i=1}^6 \varepsilon_i q_i^k=0.
\]

If there were a nontrivial linear relation among a proper subcollection of the
\(q_i^k\)'s, choose one with the smallest possible number \(s\) of terms.  The cases
\(s=1\) and \(s=2\) are impossible, since \(s=2\) would force two of the \(q_i\)'s to be
proportional over \(\F^\times\).  Thus \(3\le s\le 5\), and Corollary~\ref{cor:VW}
gives
\[
        k<s(s-2)\le 15,
\]
contradicting \(k\ge 18\). Hence the six-term relation is
minimal. In particular, the five polynomials
\[
        q_i^k\qquad (i\neq i_0)
\]
are linearly independent over \(\F\).

Now set
\[
        y_0=-\varepsilon_{i_0}q_{i_0}^k,
        \qquad
        y_i=\varepsilon_i q_i^k\quad (i\neq i_0),
\]
renumbering the five \(y_i\)'s as \(y_1,\ldots,y_5\). Then
\[
        y_1+\cdots+y_5=y_0
\]
and
\[
        \deg y_0=kT.
\]
Moreover,
\[
        \gcd(y_1,\ldots,y_5)=1.
\]
Indeed, any common divisor of \(y_1,\ldots,y_5\) also divides their sum \(y_0\),
and hence divides all six polynomials \(q_i^k\). Since
\[
        \gcd(q_1,\ldots,q_6)=1,
\]
this common divisor must be constant. By  Lemma~\ref{lem:VW-independent}, 
\begin{equation}\label{eq:ub}
 kT
 \leq
 -20+\sum_{\alpha\in \F \cup \{\infty\}}
 \Delta_\alpha(y_1,\ldots,y_5).
\end{equation}

We first estimate the finite-place contribution. For \(\alpha\in\F\), put
\[
 s_\alpha=\sum_{i=1}^6 v_\alpha(q_i).
\]
Since \(\gcd(q_1,\ldots,q_6)=1\), at least one of the six valuations is zero. The product relation gives
\[
 v_\alpha(q_1)+v_\alpha(q_2)+v_\alpha(q_3)
 =
 v_\alpha(q_4)+v_\alpha(q_5)+v_\alpha(q_6),
\]
so \(s_\alpha\) is even. Next we prove the following claim.

\begin{claim}\label{claim:salpha}
We have 
\[
\Delta_\alpha(y_1,\ldots,y_5)\leq 3s_\alpha.
\]
\end{claim}
\begin{poc}
We consider the following three cases.

\textbf{Case 1: \(s_\alpha=0\).} Then all six finite valuations are zero, so
\[
 \Delta_\alpha(y_1,\ldots,y_5)=0.
\]

\textbf{Case 2: \(s_\alpha=2\).} Then, after reordering, the valuation pattern of the \(q_i\)'s is
\[
        0,0,0,0,1,1 .
\]
Thus the six valuations of \(y_0,y_1,\ldots,y_5\) at \(\alpha\) are
\[
        0,0,0,0,k,k .
\]
After discarding the smallest valuation in the definition of
\(\Delta_\alpha(y_1,\ldots,y_5)\), we obtain
\[
        \Delta_\alpha(y_1,\ldots,y_5)=\Delta(0,0,0,k,k).
\]
By definition, this amounts to choosing five pairwise distinct integers with lower bounds
\(0,0,0,k,k\) so as to minimize the total increase. Since \(k\ge 18\), the minimum is attained by
\[
        0,1,2,k,k+1,
\]
whose total increase is \(0+1+2+0+1=4\). Hence
\[
        \Delta_\alpha(y_1,\ldots,y_5)
        =\binom52-4
        =6
        =3s_\alpha .
\]

\textbf{Case 3: \(s_\alpha\geq 4\).} Then the trivial bound gives
\[
 \Delta_\alpha(y_1,\ldots,y_5)\leq \binom{5}{2}=10\leq 3s_\alpha. \qedhere
\]
\end{poc}

By Claim~\ref{claim:salpha}, we have
\begin{equation}\label{eq:finite}
\sum_{\alpha\in\F}\Delta_\alpha(y_1,\ldots,y_5)
 \leq
 3\sum_{\alpha\in\F}s_\alpha
 =
 3\sum_{i=1}^6\deg q_i
 \leq
 18T.    
\end{equation}

Finally, at the infinite place, we use the trivial bound
\begin{equation}\label{eq:infinite}
 \Delta_\infty(y_1,\ldots,y_5)\leq \binom{5}{2}=10.
\end{equation}
Comparing inequalities~\eqref{eq:ub},~\eqref{eq:finite}, and~\eqref{eq:infinite}, we obtain that 
\[
 kT
 \leq
 -20+18T+10
 =
 18T-10,
\]
contradicting the assumption that \(k\geq 18\). 
\end{proof}

\begin{rem}
The only place where the threshold \(k\geq18\) enters is Proposition~\ref{prop:six-singleton}.
It is natural to ask whether the six-term relation
\[
        q_1^k+q_2^k+q_3^k=q_4^k+q_5^k+q_6^k,\qquad
        q_1q_2q_3=\lambda q_4q_5q_6,
\]
satisfies the stronger bound \(k<15\). Such a result would lower the
threshold on $k$ in our main theorem from \(18\) to \(15\). 
\end{rem}

\subsection{Proof of Proposition~\ref{prop:18}}
Let $S_3$ be the group of all permutations of $\{1,2,3\}$. Let $A_3$ be the alternating subgroup of \(S_3\). Put
\[
 \theta=(12).
\]

\begin{proof}[Proof of Proposition~\ref{prop:18}]

Suppose otherwise that there exist pairwise
distinct $a_1,a_2,a_3,b_1,b_2,b_3\in A$
such that
\[
 a_i b_j+n\neq0\qquad(1\le i,j\le3)
\]
and
\[
 \frac{P_{\mathrm{id}}}{P_\theta}=\frac{(a_1b_1+n)(a_2b_2+n)}
      {(a_1b_2+n)(a_2b_1+n)}\notin\F^\times.
\]
Here, for each $\sigma \in S_3$, write
\[
P_\sigma=\prod_{i=1}^3(a_i b_{\sigma(i)}+n).
\]
We will derive a contradiction by establishing a few claims.

First we deduce some basic relations among $P_{\sigma}$'s.

\begin{claim}\label{lem:two-term-opposite-parity}
\begin{enumerate}
    \item We have
    \[
 \prod_{\sigma\in A_3}P_\sigma
 =
 \prod_{\sigma\notin A_3}P_\sigma,
\]
\item If \(\sigma,\tau\in S_3\) have opposite signs, then
\[
P_\sigma\neq P_\tau.
\]
\end{enumerate}

\end{claim}

\begin{poc}
(1) The identity follows since both sides are equal to
\[
 \prod_{i=1}^3\prod_{j=1}^3(a_i b_j+n).
\]

(2) Since \(\sigma\) and \(\tau\) have opposite signs, \(\sigma^{-1}\tau\) is an
odd permutation of \(S_3\), and hence is a transposition. Thus there exist
distinct \(r,s\in\{1,2,3\}\), with remaining index \(t\), such that
\[
\tau(r)=\sigma(s),\qquad \tau(s)=\sigma(r),\qquad \tau(t)=\sigma(t).
\]
Put
\[
u=\sigma(r),\qquad v=\sigma(s).
\]
Then
\[
\begin{aligned}
P_\sigma-P_\tau
&=
(a_t b_{\sigma(t)}+n)
\Bigl((a_r b_u+n)(a_s b_v+n)
      -(a_r b_v+n)(a_s b_u+n)\Bigr)  \\
&=
n(a_r-a_s)(b_u-b_v)(a_t b_{\sigma(t)}+n).
\end{aligned}
\]
This is nonzero because \(n\ne0\), the six chosen polynomials are pairwise
distinct, and \(a_t b_{\sigma(t)}+n\ne0\). 
\end{poc}

Partition \(S_3\) into equivalence classes by declaring
\[
    \sigma\sim\tau
    \quad\Longleftrightarrow\quad
    P_\sigma/P_\tau\in\F^\times .
\]
For a subset
\(I\subset S_3\), write
\[
\Delta_I(P)=\sum_{\sigma\in I}\operatorname{sgn}(\sigma)P_\sigma.
\]
We show in the following claim that $\Delta_U(P)=0$ for each equivalence class $U \subseteq S_3$.

\begin{claim}\label{prop:main-obstruction}
Every equivalence class \(U\subseteq S_3\) satisfies
\[
   \Delta_U(P)=0.
\]
\end{claim}

\begin{poc}
Since \(A\) is a Diophantine tuple with property \(D_k(n)\), for each \(1\le i,j\le 3\), we can write
\[
    a_i b_j+n=f_{ij}^k
\]
for some \(f_{ij}\in\F[x]\); since \(a_i b_j+n\neq 0\), each \(f_{ij}\) is nonzero.

Let
\[
    M=(f_{ij}^k)_{1\le i,j\le 3}.
\]
Then the rank of $M$ is at most $2$ and hence \(\det M=0\). Expanding the determinant of $M$ gives
\[
    \sum_{\sigma\in S_3}\sgn(\sigma)
    \prod_{i=1}^3 f_{i,\sigma(i)}^k=0.
\]
For each $\sigma \in S_3$, set
\[
    Q_\sigma=\prod_{i=1}^3 f_{i,\sigma(i)}.
\]
Then \(Q_\sigma\neq 0\), \(Q_\sigma^k=P_\sigma\), and
\begin{equation}\label{eq:det-power-relation}
    \sum_{\sigma\in S_3}\sgn(\sigma)Q_\sigma^k=0.
\end{equation}

The equivalence classes of the \(Q_\sigma\)'s are the same as those of the
\(P_\sigma\)'s. Indeed, if \(Q_\sigma/Q_\tau\in\F^\times\), then clearly
\(P_\sigma/P_\tau\in\F^\times\). Conversely, if
\(P_\sigma/P_\tau\in\F^\times\), then
\[
    (Q_\sigma/Q_\tau)^k\in\F^\times.
\]
Since \(Q_\sigma/Q_\tau\in\F(x)\), this implies
\(Q_\sigma/Q_\tau\in\F^\times\).

Suppose, for contradiction, that some equivalence class has nonzero signed sum. For an equivalence class \(U\), choose a representative \(R_U\).  Then each
\(Q_\sigma\) with \(\sigma\in U\) has the form
\[
        Q_\sigma=c_\sigma R_U,\qquad c_\sigma\in\mathbb F^\times.
\]
Thus the collapsed coefficient of \(R_U^k\) is
\[
        \sum_{\sigma\in U}\operatorname{sgn}(\sigma)c_\sigma^k,
\]
which is nonzero exactly when \(\Delta_U(P)\ne0\). Collapse the determinant relation~\eqref{eq:det-power-relation} by equivalence classes, and discard the
classes with zero collapsed coefficient. This gives a nontrivial relation among \(t\)
pairwise nonproportional \(k\)-th powers, where \(t\) is the number of equivalence
classes with nonzero collapsed coefficient. 

The cases \(t=1\) and \(t=2\) are impossible: a nontrivial one-term relation
cannot occur, and a two-term relation would imply that two representatives are
proportional over \(\F^\times\), contradicting the definition of distinct
equivalence classes. Thus, if \(t\leq 5\), then \(3\leq t\leq 5\), and
Corollary~\ref{cor:VW} gives
\[
        k<t(t-2)\leq 15,
\]
contradicting \(k\geq 18\). 

Thus \(t=6\). Hence all six equivalence classes are singletons, so the \(Q_\sigma\)'s are pairwise
nonproportional. By Claim~\ref{lem:two-term-opposite-parity}(1),
\[
        \left(
        \frac{\prod_{\sigma\in A_3}Q_\sigma}
             {\prod_{\sigma\notin A_3}Q_\sigma}
        \right)^k
        =
        \frac{\prod_{\sigma\in A_3}P_\sigma}
             {\prod_{\sigma\notin A_3}P_\sigma}
        =
        1.
\]
It follows that 
\[
        \prod_{\sigma\in A_3}Q_\sigma
        =
        \lambda
        \prod_{\sigma\notin A_3}Q_\sigma
\]
for some \(\lambda\in\F^\times\). Since we have equation~\eqref{eq:det-power-relation}, Proposition~\ref{prop:six-singleton} implies that $k<18$, again contradicting the assumption \(k\geq 18\). 
\end{poc}

We now obtain a contradiction by comparing Claim~\ref{prop:main-obstruction} with the following claim.

\begin{claim}\label{lem:global-class}
If
\[
 P_\theta/P_{\mathrm{id}}\notin \F^\times,
\]
then there is an equivalence class \(U\subset S_3\) such that
\[
 \Delta_U(P)\neq 0.
\]
\end{claim}

\begin{poc}
Suppose otherwise that \(\Delta_U(P)=0\) for every equivalence class
\(U\). Then no class is a singleton. Also, if $U$ is an equivalence class of size $2$, then since \(\Delta_U(P)=0\), by Claim~\ref{lem:two-term-opposite-parity}(2), $U$ has to consist of two permutations of the same sign.

Since \(\theta\not\sim \mathrm{id}\), there are at least two distinct equivalence classes. Thus the size of the partition of $S_3$ into equivalence classes is one of
\[
 3+3,\qquad 2+4,\qquad 2+2+2.
\]
The case \(2+2+2\) is impossible, because three pairs cannot partition three
even and three odd permutations without producing an opposite-parity pair.

Therefore the partition has two classes, say \(V\) and \(W\) with $|V|\leq |W|$. Choose
representatives \(R\) and \(S\) for the two classes. Let
\[
 \delta=|V\cap A_3|-|V\setminus A_3|.
\]
Writing every element of \(V\) as a scalar multiple of \(R\), and every element
of \(W\) as a scalar multiple of \(S\), the product identity in Claim~\ref{lem:two-term-opposite-parity}(1)
gives, up to a scalar in \(\mathbb F^\times\),
\begin{equation}\label{eq:ratio}
\frac{R^{|V \cap A_3|} S^{|W \cap A_3|}}{R^{|V \setminus A_3|}S^{|W \setminus A_3|}}=\bigg(\frac{R}{S}\bigg)^\delta\in \F^\times.
\end{equation}
If $|V|=|W|=3$, then
\(\delta=2|V\cap A_3|-3\neq0\); if $|V|=2$ and $|W|=4$, by the preceding paragraph, the two elements of $V$ have the
same sign, so \(\delta=\pm2\). Hence \(\delta\neq0\) in both cases. Since
\(R/S\in\F(x)\), it follows from equation~\eqref{eq:ratio} that \(R/S\in\F^\times\). Thus, \(V\) and \(W\) were
the same equivalence class after all, a contradiction.
\end{poc}

Since \(P_\theta/P_{\mathrm{id}}\notin \F^\times\) by assumption, Claim~\ref{lem:global-class} gives
an equivalence class \(U\) with
\[
        \Delta_U(P)\neq 0.
\]
This contradicts Claim~\ref{prop:main-obstruction}. Therefore the assumed configuration cannot exist,
and Proposition~\ref{prop:18} follows.
\end{proof}

\section{Proof of Theorem~\ref{thm:main}: the nonsquare case}\label{sec:thm1}

In this section, we work on the nonsquare case.

\begin{lem}\label{lem:same-side}
Let \(n\in\F[x]\setminus\{0\}\) be nonsquare. Choose \(\alpha\in\F\) such
that \(m=v_\alpha(n)\) is odd. Let \(T\subset\F[x]\) be a set of size at least $4$ such that either
\[
 v_\alpha(t)\ge (m+1)/2
\]
for all $t\in T$, or
\[
 v_\alpha(t)\le (m-1)/2
\]
for all $t\in T$. If \(x_1,x_2,y_1,y_2\in T\) are pairwise distinct, then
\[
 x_i y_j+n\neq0\qquad(1\le i,j\le2),
\]
and
\[
 \frac{(x_1y_1+n)(x_2y_2+n)}
      {(x_1y_2+n)(x_2y_1+n)}
 \notin\F^\times.
\]
\end{lem}

\begin{proof}
The nonvanishing of $x_iy_j+n$ is immediate from the defining inequalities for \(T\): in
the first case \(v_\alpha(x_i y_j)>v_\alpha(n)\), and in the second case
\(v_\alpha(x_i y_j)<v_\alpha(n)\).

Put
\[
 N=(x_1y_1+n)(x_2y_2+n),\qquad
 M=(x_1y_2+n)(x_2y_1+n).
\]
Then
\[
 N-M=n(x_1-x_2)(y_1-y_2)\neq0.
\]

In the high-valuation case (that is, $v_{\alpha}(t)\geq (m+1)/2$ for all $t\in T$), we can write
\[
 N=n^2+E_N,\qquad M=n^2+E_M,
\]
with $v_\alpha(E_N),v_\alpha(E_M)>2m$;
in the low-valuation case (that is, $v_{\alpha}(t)\leq (m-1)/2$ for all $t\in T$), with \(H=x_1x_2y_1y_2\), we have
\[
 N=H+E_N,\qquad M=H+E_M,
\]
with $v_\alpha(E_N),v_\alpha(E_M)>v_\alpha(H).$
In both cases we have \(E_N\neq E_M\) since \(N\neq M\). Lemma~\ref{lem:valuation-ratio}
therefore gives \(N/M\notin\F^\times\), as required.
\end{proof}

Now we are ready to present the proof of Theorem~\ref{thm:main} for nonsquare $n$.

\begin{thm}\label{thm1}
Let $n\in \F[x]\setminus\{0\}$ be nonsquare, and let $A\subset \F[x]$ be a Diophantine tuple with property $D_k(n)$. 
\begin{enumerate}
    \item If $k\geq 18$, then
\[
|A|\leq 6;
\]
\item If $k\geq 18$ and $k$ is even, then we have the stronger bound
\[
|A|\leq 5.
\]
\end{enumerate} 
\end{thm}

\begin{proof}
Since $n$ is a nonsquare, we can choose \(\alpha\in\F\) such that \(m=v_\alpha(n)\) is odd. Set 
\[
 A^+=\{a\in A:v_\alpha(a)\geq (m+1)/2\},
 \qquad
 A^-=\{a\in A:v_\alpha(a)\leq (m-1)/2\}.
\]

(1) Suppose \(|A|\ge7\). By pigeonhole, one of \(A^+\) and \(A^-\) contains four elements. Thus, by 
Lemma~\ref{lem:same-side}, we can choose pairwise
distinct \(a_1,a_2,b_1,b_2\in A\) with \(a_i b_j+n\neq0\) for $1\leq i,j\leq 2$ and with 
\[
 \frac{(a_1b_1+n)(a_2b_2+n)}
      {(a_1b_2+n)(a_2b_1+n)} \notin \F^\times.
\]
Lemma~\ref{lem:extension-no-C} gives
\(a_3,b_3\in A\) such that the resulting six elements are pairwise distinct
and all \(a_i b_j+n\) are nonzero for $1\leq i,j\leq 3$. This contradicts Proposition~\ref{prop:18}.

(2) Suppose \(|A|\ge6\). Since \(k\) is even, we have $|A^+|\le1.$ Indeed, two distinct elements \(a,b\in A^+\) would give
\(v_\alpha(ab+n)=m\), contradicting that \(ab+n\) is a square. Thus
\(|A^-|\ge5\).

If \(|A^-|\ge6\), choose six distinct elements from \(A^-\), naming them
\(a_1,a_2,a_3,b_1,b_2,b_3\). By Lemma~\ref{lem:same-side}, all \(a_i b_j+n\) are nonzero for $1\leq i,j\leq 3$ and \[
 \frac{(a_1b_1+n)(a_2b_2+n)}
      {(a_1b_2+n)(a_2b_1+n)}
 \notin\F^\times.
\]
This contradicts
Proposition~\ref{prop:18}.

It remains to consider the case \(|A^-|=5\). Then \(A=A^-\cup\{h\}\) with
\(h\in A^+\). By Lemma~\ref{lem:same-side}, we have $aa'+n\neq 0$ for distinct $a,a'\in A^-$. Note that there is at most one \(\ell\in A^-\) that satisfies \(h\ell+n=0\). Thus, we can choose
three elements \(b_1,b_2,b_3\in A^-\) with \(hb_j+n\neq0\) for $1\leq j \leq 3$, and call the two
remaining elements of \(A^-\) by \(a_1,a_2\). Set \(a_3=h\). Then all \(a_i b_j+n\) are nonzero for $1\leq i,j\leq 3$. Applying Lemma~\ref{lem:same-side} to
\(a_1,a_2,b_1,b_2 \in A^-\) gives
\[
 \frac{(a_1b_1+n)(a_2b_2+n)}
      {(a_1b_2+n)(a_2b_1+n)}
 \notin\F^\times.
\]
Again this contradicts Proposition~\ref{prop:18}. 
\end{proof}

\section{Proof of Theorem~\ref{thm:main}: the square case}\label{sec:thm2}
In this section, we work on the square case \(n=s^2\). 

Recall that in the nonsquare case an odd valuation of $n$ partitions the tuple into high-valuation and low-valuation classes; in the square case all valuations of n are even, so this dichotomy no longer forces a nonconstant cross ratio.

We instead split into two cases: \(A\) is contained in an affine \(\mathbb F\)-line, or \(A\) is not contained
in an affine \(\mathbb F\)-line. Here \emph{an affine $\F$-line in $\F[x]$} is a set of the form
$a_0+\F h=\{a_0+ch: c \in \F\}$, where $a_0\in\F[x]$ and
$h\in\F[x]\setminus\{0\}$. Note that the case in which \(A\) is contained in an affine $\F$-line has been essentially handled in Section~\ref{subsec:CN}.

\subsection{The exceptional set \(A\cap s\F\)}
\begin{lem}\label{lem:sF-small}
Let \(s\in\F[x]\setminus\{0\}\), \(n=s^2\), and let \(A\subset\F[x]\) be a
Diophantine tuple with property \(D_k(n)\), where \(k\ge18\). If \(n\) is a
\(k\)-th power, assume additionally that \(A\not\subset s\F\). Then
\[
        |A\cap s\F|\le 2.
\]
\end{lem}

\begin{proof}
Assume first that $n$ is not a $k$-th power. Suppose $|A\cap s\F|\ge 3$. If $0\in A\cap s\F$, then there is a nonzero $cs\in A\cap s\F$, and
\[
 0\cdot cs+n=n
\]
would be a $k$-th power, a contradiction. Hence there are three distinct nonzero constants $c_1,c_2,c_3\in \F^\times$ such that $c_1s,c_2s,c_3s\in A$. At least one of
\[
 c_1c_2+1,
 \qquad
 c_1c_3+1,
 \qquad
 c_2c_3+1
\]
is nonzero; otherwise $c_1c_2=c_1c_3=c_2c_3=-1$, forcing $c_2=c_3$, a contradiction. After relabeling, assume $c_1c_2+1\neq 0$. Then
\[
 (c_1s)(c_2s)+s^2=(c_1c_2+1)n
\]
is a $k$-th power. Since $\F$ is algebraically closed and $c_1c_2+1\neq 0$, this implies that $n$ itself is a $k$-th power, a contradiction.

Now assume that $n$ is a $k$-th power. Then there exists $q\in \F[x]\setminus \{0\}$ such that $n=s^2=q^k$. By hypothesis $A\not\subset s\F$, so choose $a\in A\setminus s\F$. If $|A\cap s\F|\ge 3$, choose two distinct nonzero constants $c_1,c_2\in \F^\times$ such that $c_1s,c_2s\in A$. Since $A$ has property $D_k(n)$, there exist $h_1,h_2\in \F[x]$ such that
\[
 ac_1s+s^2=h_1^k,
 \qquad
 ac_2s+s^2=h_2^k.
\]
The polynomials \(h_1,h_2\) are nonzero. Indeed, if \(h_i=0\) for some
\(i\in\{1,2\}\), then
\[
        ac_i s+s^2=0,
\]
so \(a=-s/c_i\in s\F\), contradicting the choice of \(a\). Observe that we have
\begin{equation}\label{eq:c1c2}
 c_2h_1^k-c_1h_2^k=(c_2-c_1)q^k.
\end{equation}
We claim that $h_1$ is not proportional to either $h_2$ or $q$. 

Suppose first that \(h_1=\lambda h_2\) for some \(\lambda\in \F^\times\). Then
\[
        ac_1s+s^2=h_1^k=\lambda^k h_2^k
        =\lambda^k(ac_2s+s^2),
\]
and hence
\[
        a(c_1-\lambda^k c_2)=(\lambda^k-1)s.
\]
Here \(\lambda^k\neq 1\), since otherwise \(h_1^k=h_2^k\), forcing \(c_1=c_2\). Therefore the
right-hand side is nonzero, and the displayed identity gives \(a\in s\F\), a contradiction.

Similarly, if $h_1=\lambda q$ for some $\lambda\in\F^\times$, then
\[
ac_1s+s^2=h_1^k=\lambda^kq^k=\lambda^ks^2,
\]
so $ac_1=(\lambda^k-1)s$, again contradicting $a\notin s\F$. Hence $h_1/q\notin\F^\times$.

By Corollary~\ref{cor:VW} applied to the three-term relation~\eqref{eq:c1c2}, we get $k<3$, contradicting the assumption $k\ge 18$. 
\end{proof}

\subsection{Selection lemma: $A$ is not contained in an affine \(\F\)-line}
In this subsection, we consider the case where $A$ is not contained in an affine \(\F\)-line. We apply a geometric argument.

\begin{lem}
\label{lem:non-affine-7}
Let \(s\in \F[x]\setminus\{0\}\), and let \(A\subset \F[x]\) be a Diophantine tuple with property
\(D_k(s^2)\), where \(k\ge 18\). Suppose that \(|A|\ge 7\) and that \(A\) is not contained in an
affine \(\F\)-line. Then there exist pairwise distinct \(a_1,a_2,b_1,b_2\in A\) such that
\[
        a_i b_j+s^2\neq 0\qquad (1\le i,j\le 2),
\]
and
\[
        \frac{(a_1b_1+s^2)(a_2b_2+s^2)}
             {(a_1b_2+s^2)(a_2b_1+s^2)}
        \notin \F^\times .
\]
\end{lem}

\begin{proof}
Suppose, for contradiction, that no such four elements exist. For distinct \(r,t\in A\), put
\[
        G(r,t)=\{w\in A\setminus\{r,t\}: rw+s^2\neq 0,\; tw+s^2\neq 0\}.
\]
We shall use repeatedly the following fact: for each \(a\in A\), the equation
\(ab+s^2=0\) has at most one solution \(b\in A\setminus\{a\}\). Indeed, two such
solutions \(b,b'\) would give \(a(b-b')=0\), while the existence of a solution already
forces \(a\neq0\).

\begin{claim}\label{claim:affineline}
For every distinct \(r,t\in A\), the set \(G(r,t)\) is contained in an affine \(\F\)-line.    
\end{claim}

\begin{poc}

Since at most two elements of \(A\setminus\{r,t\}\) are excluded, we have
\[
        |G(r,t)|\ge |A|-4\ge 3.
\]
It is enough to show that any three elements of \(G(r,t)\) are collinear over \(\F\). 

Let
\(w,z\in G(r,t)\) be distinct. By our assumption to the contrary,
\[
        C=\frac{(rw+s^2)(tz+s^2)}{(rz+s^2)(tw+s^2)}\in \F^\times.
\]
Moreover \(C\neq 1\), because
\[
        (rw+s^2)(tz+s^2)-(rz+s^2)(tw+s^2)=s^2(r-t)(w-z)\neq 0.
\]
Thus, with \(H=s^2(r-t)\), we have
\[
        H(w-z)=(C-1)(rz+s^2)(tw+s^2).
\]
The right-hand side is a nonzero scalar multiple of a \(k\)-th power, and since \(\F\) is
algebraically closed, there is \(g\in \F[x]\setminus\{0\}\) such that
\[
        H(w-z)=g^k .
\]

Now take distinct \(w_1,w_2,w_3\in G(r,t)\). Then
\[
        H(w_1-w_2)=g_{12}^k,\qquad
        H(w_1-w_3)=g_{13}^k,\qquad
        H(w_2-w_3)=g_{23}^k,
\]
and hence
\[
        g_{12}^k-g_{13}^k+g_{23}^k=0.
\]
By Corollary~\ref{cor:VW}, the polynomials \(g_{12},g_{13},g_{23}\) must all be
proportional over \(\mathbb F^\times\); otherwise one of them would be nonproportional to the other two, giving \(k<3\). Therefore \(w_1-w_2,w_1-w_3,w_2-w_3\) are proportional over \(\F^\times\), so
\(w_1,w_2,w_3\) lie on an affine \(\F\)-line. 

Fix two distinct elements \(a,b\in G(r,t)\). By the above argument, every third element of \(G(r,t)\)
is collinear with \(a\) and \(b\). The whole set \(G(r,t)\) is contained in the affine
\(\F\)-line through \(a\) and \(b\).
\end{poc}
Choose distinct \(u,v\in A\) such that \(uv+s^2=0\), if such a pair exists; otherwise choose
arbitrary distinct \(u,v\in A\). Then
\[
        G(u,v)=A\setminus\{u,v\}.
\]
Indeed, this is immediate if no pair in \(A\) has product \(-s^2\); and if \(uv+s^2=0\), then the
uniqueness observation shows that no \(w\in A\setminus\{u,v\}\) is excluded from \(G(u,v)\).
Thus \(|G(u,v)|=|A|-2\ge 5\), and by Claim~\ref{claim:affineline}, \(G(u,v)\) is contained in an affine \(\F\)-line, say \(L\).

Choose distinct \(c,d\in G(u,v)\) such that \(cd+s^2=0\), if such a pair exists inside \(G(u,v)\);
otherwise choose arbitrary distinct \(c,d\in G(u,v)\). We claim that
\[
        \{u,v\}\cup\bigl(G(u,v)\setminus\{c,d\}\bigr)\subset G(c,d).
\]
Indeed, \(u,v\in G(c,d)\) because \(c,d\in G(u,v)\). If \(z\in G(u,v)\setminus\{c,d\}\), then
\(cz+s^2\) and \(dz+s^2\) are nonzero: this follows either from the choice \(cd+s^2=0\) and
uniqueness, or from the fact that no pair inside \(G(u,v)\) has product \(-s^2\). Hence
\(z\in G(c,d)\).

Since \(|G(u,v)\setminus\{c,d\}|\ge 3\), choose distinct
\(z_1,z_2\in G(u,v)\setminus\{c,d\}\). Then \(u,v,z_1,z_2\in G(c,d)\). By Claim~\ref{claim:affineline}, \(G(c,d)\)
is contained in an affine \(\F\)-line. This line contains \(z_1,z_2\in L\), so it is \(L\). Hence
\(u,v\in L\). Therefore
\[
        A=\{u,v\}\cup G(u,v)\subset L,
\]
contradicting the assumption that \(A\) is not contained in an affine \(\F\)-line.
\end{proof}

\subsection{Finishing the proof }
We conclude the paper with a proof of Theorem~\ref{thm:main} for square $n$.

\begin{thm}\label{thm2}
Let $s\in \F[x] \setminus \{0\}$ and $n=s^2$. Let $A\subset \F[x]$ be a Diophantine tuple with property $D_k(n)$. If $n$ is a $k$-th power in $\F[x]$, assume additionally that $A\not\subset s\F$. If $k\ge 18$, then
\[
 |A|\le 6.
\]
\end{thm}

\begin{proof}
Suppose, for contradiction, that
\[
 |A|\geq 7.
\]
We first note that \(A\not\subset s\F\). Indeed, if \(n\) is a \(k\)-th power in \(\F[x]\), this is part of the hypothesis. If \(n\) is not a \(k\)-th power, then Lemma~\ref{lem:sF-small} gives \(|A\cap s\F|\le 2\), while \(|A|\ge 7\), and hence \(A\not\subset s\F\).

We now find pairwise distinct
\[
 a_1,a_2,b_1,b_2\in A
\]
such that
\[
 a_i b_j+s^2\neq 0\qquad (1\leq i,j\leq 2),
\]
and
\[
 \frac{(a_1b_1+s^2)(a_2b_2+s^2)}
        {(a_1b_2+s^2)(a_2b_1+s^2)}
 \notin \F^\times.
\]
If \(A\) is contained in an affine
\(\F\)-line, this follows from Corollary~\ref{cor:affine}; otherwise, this follows from
Lemma~\ref{lem:non-affine-7}. 

Now Lemma~\ref{lem:extension-no-C} gives
\(a_3,b_3\in A\) such that the resulting six elements are pairwise distinct
and all \(a_i b_j+n\) are nonzero for $1\leq i,j\leq 3$. This contradicts Proposition~\ref{prop:18}.
\end{proof}

\section*{Acknowledgments}
The authors thank Ernie Croot, Andrej Dujella, Seoyoung Kim, and Th\'{a}i Ho\`ang L\^{e} for helpful discussions. 

\bibliographystyle{abbrv}
\bibliography{references}

\end{document}